\documentclass{amsart}

\usepackage{amssymb,amsmath,amsthm,latexsym,stmaryrd}

\newtheorem{thm}{Theorem}[section]
\newtheorem{theorem}{Theorem}[section]
\newtheorem*{theorem A}{Theorem A}
\newtheorem*{theorem B}{N\"olker's Theorem}

\theoremstyle{remark}

\theoremstyle{remark}

\theoremstyle{definition}

\numberwithin{equation}{section}
\def\({\left ( }
    \def\){\right )}
\def\<{\left < }
\def\>{\right >}

\newcommand{\ie}{i.e. }

\newcommand{\s}{\mathfrak{S}}
\newcommand{\HCC}{\mathcal{HC}}
\newcommand{\be}[1]{\begin{equation}\label{#1}}
\newcommand{\ee}{\end{equation}}
\newcommand{\n}{\nabla}

\newcommand{\z}{\mathfrak{z}}
\newcommand{\W}{\mathcal{W}}
\newcommand{\MM}{\mathcal{M}}
\newcommand{\ea}{\varepsilon_\alpha}

\newcommand{\KKK}{\mathcal{K}}
\newcommand{\R}{\mathbb R}
\newcommand{\C}{\mathbb C}
\newcommand{\X}{\mathfrak X}

\newcommand{\ta}{\theta}

\newcommand{\lm}{\lambda}
\newcommand{\gm}{\gamma}
\newcommand{\al}{\alpha}
\newcommand{\bt}{\beta}
\newcommand{\Dt}{\Delta}

\newcommand{\RNum}[1]{\uppercase\expandafter{\romannumeral #1\relax}}
\newcommand{\Bia}[1]{\mathrm{Bia}(\mathrm{\uppercase\expandafter{\romannumeral #1\relax}})}

\newcommand{\thmref}[1]{Theorem~\ref{#1}}

\begin{document}

\vspace{2cm}

\title[Matrix Lie Groups as 4-Dimensional Hypercomplex Manifolds \dots]
{Matrix Lie Groups as 4-Dimensional Hypercomplex Manifolds with Hermitian-Norden metrics}

\author{Hristo Manev}
\address{Medical University -- Plovdiv, Faculty of Public Health,
Department of Medical Informatics, Biostatistics and e-Learning,   15-A Vasil Aprilov
Blvd.,   Plovdiv 4002,   Bulgaria;}
\email{hristo.manev@mu-plovdiv.bg}





\begin{abstract}
There are studied Lie groups considered as almost hypercomplex Hermitian-Norden manifolds, which are integrable and have the lowest dimension four.
It is established a correspondence of the derived Lie algebras of types of invariant hypercomplex structures and the explicit matrix representation of their Lie groups. There are constructed examples of the considered structure of different types on some known Lie groups.\thanks{Research supported by National Scientific Program ''Young Researchers and Post-Doctorants''}
\end{abstract}
\maketitle
\section*{Introduction}

In this work we continue the investigation of the almost hypercomplex manifold with Hermitian-Norden  metrics. The manifold under consideration is a $4n$-dimensional manifold equipped with an almost hypercomplex structure $H$ which is a triad of anticommuting almost complex structures such that each of them is a composition of the two other structures. Moreover, the introduced pseudo-Riemannian metric on the manifold is Hermitian with respect to one of the structures of $H$ and a Norden metric with respect to other two structures of $H$.

The geometry of almost hypercomplex manifolds with Hermitian-Norden metrics is studied in \cite{GriManDim12,Ma05,ManSek,GriMan24,Man28,ManGri32,NakHMan}. This type of manifolds are the only possible way to involve Norden-type metrics on almost hypercomplex manifolds, which origin from complex Riemannian manifolds.

In \cite{Barb,GriMan24,Man44,HM5,ZamNak}, there are considered Lie groups as manifolds with different similar structures and metrics. Moreover, in \cite{Man44}, integrable hypercomplex structures with Hermitian-Norden metrics on Lie groups of dimension 4 are considered. There are constructed the corresponding five types of invariant hypercomplex structures. The different cases regarding the signature of the metric are considered.

As it is known from \cite{Gil}, each representation of a Lie algebra corresponds uniquely to a representation of a simply connected Lie group. This relation is one-to-one. Thus, knowing the representation of a certain Lie algebra settles the question of the representation of its Lie group.
In this work it is established a correspondence of the derived Lie algebras of the types of invariant hypercomplex structures, constructed in \cite{Man44}, and the explicit matrix representation of their Lie groups.

The last section of the work is devoted to equipping known Lie groups with integrable hypercomplex structures with Hermitian-Norden metrics of different types.

\section{Preliminaries}\label{sect-prel}

\subsection{Almost hypercomplex manifolds with Hermitian-Norden metrics}

Firstly, let us recall the notion of almost hypercomplex structure $H$
on a $4n$-dimensional
manifold $\MM^{4n}$. The considered structure is a triad
of anticommuting almost complex structures
such that each of them is a composition of the two other structures, \ie for $H=(J_\alpha)$ $(\alpha=1,2,3)$ the property $J_3=J_1\circ J_2=-J_2\circ J_1$ is satisfied
\cite{AlMa,So}.

Let us equip $H$ with a metric structure,
generated by a pseudo-Riemannian metric $g$ of neutral signature
\cite{GriMan24,GriManDim12}. Here, in the considered case, one (resp., the other
two) of the almost complex structures of $H$ acts as an isometry
(resp., anti-isometries) with respect to $g$ in each
tangent fibre.
Thus, there exist three (0,2)-tensors
associated with $H$ to the metric $g$ -- a K\"ahler form and two metrics of the Hermitian-Norden type.
In the considered case we have a Hermitian metric $g$ with respect to $J_1$ and a Norden
metric $g$ regarding $J_2$ and $J_3$. So, we call the derived almost hypercomplex
metric structure an almost hypercomplex structure with Hermitian-Norden metrics.

Let $(\MM,H)$ be an almost hypercomplex manifold, where $\MM$ is a
$4n$-dimension\-al differentiable manifold and $H=(J_1,J_2,J_3)$
is an almost hypercomplex structures on $\MM$ with the following
properties for all cyclic permutations $(\al, \bt, \gm)$ of
$(1,2,3)$:%
\begin{equation*}\label{J123} %
J_\al=J_\bt\circ J_\gm=-J_\gm\circ J_\bt, \qquad
J_\al^2=-I,
\end{equation*} %
where $I$ denotes the identity. 

Let $g$ be a neutral metric on $(\MM,H)$ with the following properties
\begin{equation}\label{gJJ} %
g(\cdot,\cdot)=\ea g(J_\al \cdot,J_\al \cdot), \end{equation} %
where
\begin{equation*}\label{epsiloni}
\ea=
\begin{cases}
\begin{array}{ll}
1, \quad & \al=1;\\
-1, \quad & \al=2;3.
\end{array}
\end{cases}
\end{equation*}

Here and further, the index $\alpha$ runs over the range
$\{1,2,3\}$ unless otherwise is stated.

The associated 2-form $g_1$ and respectively the
associated neutral metrics $g_2$ and $g_3$ are determined by
\begin{equation}\label{gJ} %
g_\al(\cdot,\cdot)=g(J_\al \cdot,\cdot)=-\ea g(\cdot,J_\al \cdot).
\end{equation}%

The structure $(H,G)=(J_1,J_2,J_3;g,g_1,g_2,g_3)$ on
$\MM^{4n}$ is called an \emph{almost hy\-per\-com\-plex
structure with Hermit\-ian-Norden metrics} and the respective manifold $(\MM,H,G)$ is called an \emph{almost hypercomplex
manifold with Hermit\-ian-Norden metrics} (\cite{GriManDim12}). 

In \cite{GriManDim12}, there are introduced the fundamental tensors of a manifold $(\MM,H,G)$ by the following three
$(0,3)$-tensors
\begin{equation*}\label{F'-al}
F_\al (x,y,z)=g\bigl( \left( \n_x J_\al
\right)y,z\bigr)=\bigl(\n_x g_\al\bigr) \left( y,z \right),
\end{equation*}
where $\n$ is the Levi-Civita connection generated by $g$.
These tensors have the basic properties
\begin{equation*}\label{FaJ-prop}
  F_{\al}(x,y,z)=-\ea F_{\al}(x,z,y)=-\ea F_{\al}(x,J_{\al}y,J_{\al}z)
\end{equation*}
and the following relations are valid for all cyclic permutations $(\al, \bt, \gm)$ of
$(1,2,3)$
\begin{equation*}\label{F1F2F3}
    F_\al(x,y,z)=F_\bt(x,J_\gm y,z)-\ea F_\gm(x,y,J_\bt z).
\end{equation*}

The corresponding Lee forms $\ta_\al$ are defined by
\begin{equation*}\label{theta-al}
\ta_\al(\cdot)=g^{kl}F_\al(e_k,e_l,\cdot)
\end{equation*}%
for an arbitrary basis $\{e_1,e_2,\dots, e_{4n}\}$ of $T_p\MM$,
$p\in \MM$.

In \cite{GriManDim12}, the so-called \emph{hyper-K\"ahler
manifolds with Hermitian-Norden metrics} are studied. They are the almost hypercomplex manifold with Hermitian-Norden metrics in the class
$\KKK$, where $\n J_\al=0$ for all $\al$. It is proved in \cite{GriMan24} that a sufficient
condition for $(\MM,H,G)$ to be in $\KKK$ is this manifold to be of
K\"ahler-type with respect to two of the three complex structures
of $H$.

Let us remark that according to \eqref{gJJ} $(\MM,J_1,g)$
is almost Hermitian manifold whereas $(\MM,J_\al,g)$, $\al=2,3$,
are almost complex manifolds with Norden metric.
The basic classes of the mentioned two types of
manifolds are given in \cite{GrHe} and \cite{GaBo}, respectively. The special class $\W_0(J_\al):$ $F_\al=0$  of the
K\"ahler-type manifolds belongs to any other class in the corresponding classification.
In the lowest 4-dimensional case, the four basic classes
of almost Hermitian manifolds with respect to
$J_1$ are restricted to two:
$\W_2(J_1)$, the class of the almost K\"ahler manifolds, and
$\W_4(J_1)$, the class of the Hermitian manifolds which, for dimension 4, are determined by:
\begin{equation*}\label{cl-H-dim4}
\begin{split}
&\W_2(J_1):\; \mathop{\s}_{x,y,z}\bigl\{F_1(x,y,z)\bigr\}=0; \\
&\W_4(J_1):\; F_1(x,y,z)=\dfrac{1}{2}
                \left\{g(x,y)\ta_1(z)-g(x,J_1y)\ta_1(J_1z)\right. \\
&\phantom{\W_4(J_1):\; F_1(x,y,z)=\quad\,}
                \left.-g(x,z)\ta_1(y)+g(x,J_1z)\ta_1(J_1y)
                \right\},
\end{split}
\end{equation*}
where $\s $ is the cyclic sum by three
arguments $x$, $y$, $z$.
The basic classes of the almost complex manifolds with Norden metric ($\al=2$ or $3$)
are determined, for dimension $4$, as follows:
\begin{equation*}\label{cl-N-dim4}
\begin{split}
&\W_1(J_\al):\; F_\al(x,y,z)=\dfrac{1}{4}\bigl\{
g(x,y)\ta_\al(z)+g(x,J_\al y)\ta_\al(J_\al z)\bigr.\\
&\phantom{\W_1(J_\al):\; F_\al(x,y,z)=\quad\,\,} %
\bigl.+g(x,z)\ta_\al(y)
    +g(x,J_\al z)\ta_\al(J_\al y)\bigr\};\\
&\W_2(J_\al):\; \mathop{\s}_{x,y,z}
\bigl\{F_\al(x,y,J_\al z)\bigr\}=0,\qquad \ta_\al=0;\\
&\W_3(J_\al):\; \mathop{\s}_{x,y,z} \bigl\{F_\al(x,y,z)\bigr\}=0.
\end{split}
\end{equation*}


Let us denote by
$\W^0=\left\{\W\;|\;\mathrm{d}\left(\theta_1\circ J_1\right)=0\right\}$
the class of the (locally) confor\-mally equivalent
$\KKK$-manifolds, where
$
\W=\W_4(J_1)\cap\W_1(J_2)\cap W_1(J_3)
$
and a conformal transformation of the metric is given by $\overline{g} = e^{2u}g$
for a differentiable function $u$ on the manifold.

It is known that an almost hypercomplex structure $H=(J_\alpha)$  is
a hypercomplex structure if the Nijenhuis tensors $[J_\al,J_\al]$, determined by
\begin{equation*}\label{N_al} %
[J_\al,J_\al](\cdot,\cdot)\allowbreak{}= \left[J_\al \cdot,J_\al \cdot
\right]
    -J_\al\left[J_\al \cdot,\cdot \right]
    -J_\al\left[\cdot,J_\al \cdot \right]
    -\left[\cdot,\cdot \right],
\end{equation*}%
vanish on $\X(\MM)$ for each $\alpha$ (e.g. \cite{Boy}).
Moreover, it is known that $H$ is
hypercomplex if and only if two of $[J_\al,J_\al]$ vanish.

Then the class of hypercomplex manifolds with Hermitian-Norden metrics is
\[
\HCC=\W_4(J_1)\cap\left(\W_1\oplus\W_2\right)(J_2)\cap\left(\W_1\oplus\W_2\right)(J_3).
\]

Let us denote by $\HCC'=\W_0(J_1)\cap\left(\W_1\oplus\W_2\right)(J_2)\cap\left(\W_1\oplus\W_2\right)(J_3)$ the respective subclass of $\HCC$.

\section{The main result}\label{sect-lie}

Let $L$ be a simply connected 4-dimensional real Lie group
admitting an invariant hypercomplex structure, \ie left translations by elements of $L$ are holomorphic with respect to
$J_{\al}$ for all $\al$. The corresponding Lie algebra of $L$ is denoted by $\mathfrak{l}$.

In \cite{Barb}, invariant hypercomplex structures $H$ on 4-dimensional
real Lie groups are classified as follows
\begin{thm}[\cite{Barb}]\label{thm-Barb}
The only 4-dimensional Lie algebras admitting an integrable
hypercomplex structure are the following types:

\emph{\textbf{(hc1)}} $\mathfrak{l}$ is Abelian;$\qquad\qquad$

\emph{\textbf{(hc2)}} $\mathfrak{l}\cong\R\oplus\mathfrak{so}(3)$;
$\qquad\qquad$

\emph{\textbf{(hc3)}} $\mathfrak{l}\cong\mathfrak{aff}(\C)$;

\emph{\textbf{(hc4)}} $\mathfrak{l}$ is the solvable Lie algebra corresponding to $\R H^4$;

\emph{\textbf{(hc5)}} $\mathfrak{l}$ is the solvable Lie algebra corresponding to $\C H^2$,
\\
where $\R\oplus\mathfrak{so}(3)$ is the Lie algebra of the Lie
groups $U(2)$ and $S^3\times S^1$; $\mathfrak{aff}(\C)$ is the Lie
algebra of the affine motion group of $\C$, the unique
4-dimensional Lie algebra carrying an Abelian hypercomplex
structure; $\R H^4$ is the real hyperbolic space; $\C H^2$ is the
complex hyperbolic space.
\end{thm}

A standard hypercomplex structure on $\mathfrak{l}$ is
defined as in \cite{So}:
\begin{equation}\label{JJJ}
\begin{array}{llll}
J_1e_1=e_2, \quad & J_1e_2=-e_1,  \quad &J_1e_3=-e_4, \quad &J_1e_4=e_3;
\\[6pt]
J_2e_1=e_3, &J_2e_2=e_4, &J_2e_3=-e_1, &J_2e_4=-e_2;
\\[6pt]
J_3e_1=-e_4, &J_3e_2=e_3, &J_3e_3=-e_2, &J_3e_4=e_1,\\[6pt]
\end{array}
\end{equation}
where $\{e_1,e_2,e_3,e_4\}$ is a basis of a 4-dimensional real Lie
algebra $\mathfrak{l}$ with center $\z$ and derived Lie algebra
$\mathfrak{l}'=[\mathfrak{l},\mathfrak{l}]$.

A pseudo-Euclidian metric $g$ of
neutral signature is introduced by
\begin{equation}\label{g}
g(x,y)=x^1y^1+x^2y^2-x^3y^3-x^4y^4,
\end{equation}
where $x(x^1,x^2,x^3,x^4)$, $y(y^1,y^2,y^3,y^4) \in \mathfrak{l}$. The metric $g$
generates an almost hypercomplex structure with Hermitian-Norden metrics on $\mathfrak{l}$, according \eqref{gJJ} and \eqref{gJ}.

In \cite{Man44}, there are considered integrable hypercomplex structures with Hermitian-Norden metrics on Lie groups of dimension 4. There are constructed and characterized different types of hypercomplex
structures on Lie algebras following the Barberis classification of invariant hypercomplex structures $H$ on 4-dimensional real Lie groups, given in \thmref{thm-Barb}. Moreover, in \cite{Man44} it is made a correspondence between the 4-dimensional real Lie groups for each class in the Barberis classification and the respective class of hypercomplex manifold with Hermitian-Norden metrics, given in Table 1. In some of the classes there are considered separately the different cases of signature of $g$ on $\mathfrak{l}'$.

\begin{table}[ht]
\centering
\begin{tabular}{|l|c|}
  \hline
  \emph{\textbf{(hc1)}} & $\KKK$ \\
  \hline
  \emph{\textbf{(hc2)}} & $\HCC$ \\
  \hline
  \emph{\textbf{(hc3.1)}} & $\HCC'$ \\
  \hline
  \emph{\textbf{(hc3.2)}} & $\W^0$ \\
  \hline
  \emph{\textbf{(hc4.1)}} & $\HCC$ \\
  \hline
  \emph{\textbf{(hc4.2)}} & $\W^0$ \\
  \hline
  \emph{\textbf{(hc5.1)}} & $\HCC$ \\
  \hline
  \emph{\textbf{(hc5.2)}} & $\HCC$ \\
  \hline
\end{tabular}
\caption{}
\end{table}

It is known (e.g. \cite{Gil}) that for a real Lie algebra of finite dimension
there is a corresponding connected simply connected Lie group, determined uniquely up to isomorphism.
It is arose the problem for determination of the Lie group, given by its explicit matrix representation,  which is isomorphic to the Lie group equipped with an integrable hypercomplex structure corresponding to each class, given in \cite{Man44}.

\begin{theorem}\label{thm:main}
Let $(L,H,G)$ be a 4-dimensional hypercomplex manifold with Her\-mit\-ian-Norden metrics. Then the compact simply connected Lie group isomorphic to $L$, both with one and the same Lie algebra, has the form
\begin{equation*}\label{eA}
  e^A=E+tA+uA^2,
\end{equation*}
where $E$ is the identity matrix and $A$ is the matrix representation of the corresponding Lie algebra. The matrix form of $A$ as well as the real parameters $t$ and $u$  for the different classes of 4-dimensional Lie algebras admitting a hypercomplex structure are the following for $a,b,c,d\in\R$:
\end{theorem}
\begin{subequations}\label{thm1}
\begin{equation*}
\begin{array}{ll}
(hc1):\; &A=\left(
      \begin{array}{cccc}
        0 & 0 & 0 & 0 \\
        0 & 0 & 0 & 0 \\
        0 & 0 & 0 & 0 \\
        0 & 0 & 0 & 0 \\
      \end{array}
    \right),
\quad
t= 1,
\quad
u= 0;
\\
(hc2):\; &A=\left(
      \begin{array}{cccc}
        0 & 0 & 0 & 0 \\
        0 & 0 & d & -c \\
        0 & -d & 0 & b \\
        0 & c & -b & 0 \\
      \end{array}
    \right),
		\quad  \Dt:=b^2+c^2+d^2
\\
&
\begin{array}{ll}
\bullet\; \text{for}\; (b,c)\neq (0,0) \qquad &t=    \Dt^{-\frac12}\sin{\sqrt{\Dt}},\\
																															&u=\Dt^{-1}\left(1-\cos{\sqrt{\Dt}}\right);  \\
\bullet\; \text{for}\; (b,c)= (0,0) \qquad & t=1,\qquad\qquad u=0;
\end{array}
\\
(hc3.1):\; &A=\left(
      \begin{array}{cccc}
        0 & d & 0 & -b \\
        0 & c & 0 & a \\
        0 & -b & 0 & -d \\
        0 & -a & 0 & c \\
      \end{array}
    \right),
		\quad \Dt:=a(a^2+c^2)
\\
&
\begin{array}{ll}
\bullet\; \text{for}\; a(a^2+c^2)\neq 0\\
\phantom{\bullet\;}
t=\Dt^{-1}
    \left\{-2ac(1-e^c \cos a)+(a^2-c^2)e^c \sin a\right\}, \\
\phantom{\bullet\;}
u= \Dt^{-1}
    \left\{a(1-e^c \cos a)+c e^c \sin a\right\};
\end{array}
\\
&
\begin{array}{lll}
\bullet\; \text{for}\; (a,c)= (0,0) \qquad &t=1,\qquad &
u=0;
\end{array}
\\
(hc3.2):\; &A=\left(
      \begin{array}{cccc}
        c & d & 0 & 0 \\
        -d & c & 0 & 0 \\
        -a & -b & 0 & 0 \\
        b & -a & 0 & 0 \\
      \end{array}
    \right),
		\quad \Dt:=d(c^2+d^2)
\\
&
\begin{array}{ll}
\bullet\; \text{for}\; d(c^2+d^2)\neq 0\\
\phantom{\bullet\;}
t=\Dt^{-1}
    \left\{-2cd(1-e^c \cos d)+(d^2-c^2)e^c \sin d\right\}, \\
\phantom{\bullet\;}
u= \Dt^{-1}
    \left\{d(1-e^c \cos d)+c e^c \sin d\right\};
\end{array}
\\
&
\begin{array}{lll}
\bullet\; \text{for}\; (c,d)= (0,0) \qquad &t=1,\qquad &
u=0;
\end{array}
\end{array}
\end{equation*}
\\
\begin{equation*}
\begin{array}{ll}
(hc4.1):\; &A=\left(
      \begin{array}{cccc}
        0 & b & c & d \\
        0 & -a & 0 & 0 \\
        0 & 0 & -a & 0 \\
        0 & 0 & 0 & -a \\
      \end{array}
    \right),
\\
&
\begin{array}{lll}
\bullet\; \text{for}\; a\neq 0 \qquad &t=a^{-1},\qquad 		&
u=a^{-2}e^{-a};  \\
\bullet\; \text{for}\; a= 0 \qquad & t=1,\qquad & u=0;
\end{array}
\\
(hc4.2):\; &A=\left(
      \begin{array}{cccc}
        -d & 0 & 0 & 0 \\
        0 & -d & 0 & 0 \\
        0 & 0 & -d & 0 \\
        a & b & c & 0 \\
      \end{array}
    \right),
\\
&
\begin{array}{lll}
\bullet\; \text{for}\; d\neq 0 \qquad &t=d^{-1},\qquad 		&
u=d^{-2}e^{-d};  \\
\bullet\; \text{for}\; d= 0 \qquad & t=1,\qquad & u=0;
\end{array}
\\
(hc5.1):\; &A=\left(
      \begin{array}{cccc}
        0 & b & \frac{c}{2} & \frac{d}{2} \\
        0 & -a & 0 & 0 \\
        0 & \frac{d}{2} & -\frac{a}{2} & 0 \\
        0 & -\frac{c}{2} & 0 & -\frac{a}{2} \\
      \end{array}
    \right),
\\
&
\begin{array}{ll}
\bullet\; \text{for}\; a\neq 0 \qquad &t=a^{-1}\left\{e^{-a}-4e^{-\frac{a}{2}}+3\right\},\\
																										&u=a^{-2}\left\{2e^{-a}-4e^{-\frac{a}{2}}+2\right\};  \\
\bullet\; \text{for}\; a= 0 \qquad & t=1,\qquad\qquad u=0;
\end{array}
\\
(hc5.2):\; &A=\left(
      \begin{array}{cccc}
        -\frac{d}{2} & 0 & -\frac{b}{2} & 0 \\
        0 & -\frac{d}{2} & \frac{a}{2} & 0 \\
        0 & 0 & -d & 0 \\
        \frac{a}{2} & \frac{b}{2} & c & 0 \\
      \end{array}
    \right),
\\
&
\begin{array}{ll}
\bullet\; \text{for}\; d\neq 0 \qquad &t=d^{-1}\left\{e^{-d}-4e^{-\frac{d}{2}}+3\right\},\\
																										&u=d^{-2}\left\{2e^{-d}-4e^{-\frac{d}{2}}+2\right\};  \\
\bullet\; \text{for}\; d= 0 \qquad & t=1,\qquad\qquad u=0.
\end{array}
\end{array}
\end{equation*}
\end{subequations}

\begin{proof}
To prove this theorem we can follow the idea of the proof of Theorem 2.1 in \cite{HM5}.
Let us consider the case when $(L,H,G)$ is a 4-dimensional hypercomplex manifold with Her\-mit\-ian-Norden metrics for which the corresponding 4-dimensional Lie algebra admitting a hypercomplex structure is of the class $(hc2)$.
Then, according to \cite{Man44}, the corresponding Lie algebra of $(hc2)$ is determined by
\begin{equation}\label{com1}
[e_2,e_4]=e_3,\quad [e_4,e_3]=e_2,\quad [e_3,e_2]=e_4.
\end{equation}
Using \eqref{com1}, we have that the nonzero values of the commutation coefficients are
\begin{equation}\label{Cij}
C_{24}^3=C_{43}^2=C_{32}^4=-C_{42}^3=-C_{34}^2=-C_{23}^4=1.
\end{equation}

According to \cite{Gil}, the commutation coefficients provide a matrix representation of the respective Lie algebra.
This representation is obtained by the basic matrices $M_k$. Their entries are determined by
\begin{equation}\label{Mij}
(M_k)_l^s=-C_{kl}^s.
\end{equation}
Using \eqref{Cij} and \eqref{Mij}, we obtain

\begin{equation*}\label{}
\begin{array}{c}
M_1=\left(
      \begin{array}{cccc}
        0 & 0 & 0 & 0\\
        0 & 0 & 0 & 0\\
        0 & 0 & 0 & 0\\
        0 & 0 & 0 & 0\\
      \end{array}
    \right),\quad
M_2=\left(
      \begin{array}{cccc}
        0 & 0 & 0 & 0\\
        0 & 0 & 0 & 0\\
        0 & 0 & 0 & 1\\
        0 & 0 & -1 & 0\\
      \end{array}
    \right),\\[30pt]
M_3=\left(
      \begin{array}{cccc}
        0 & 0 & 0 & 0\\
        0 & 0 & 0 & -1\\
        0 & 0 & 0 & 0\\
        0 & 1 & 0 & 0\\
      \end{array}
    \right),\quad
M_4=\left(
      \begin{array}{cccc}
        0 & 0 & 0 & 0\\
        0 & 0 & 1 & 0\\
        0 & -1 & 0 & 0\\
        0 & 0 & 0 & 0\\
      \end{array}
    \right).
\end{array}
\end{equation*}
Then we have $A=aM_1+bM_2+cM_3+dM_4$ for $a,b,c,d\in\R$.

Let us suppose that $(b,c,d)\neq (0,0,0)$.
The matrix representation of the considered Lie algebra is the matrix $A$, given in \thmref{thm:main}.
Then, the characteristic polynomial of $A$ is
\begin{equation*}\label{}
P_A(\lm)= \lm^2(\lm^2 + b^2 + c^2 + d^2).
\end{equation*}
Therefore, the eigenvalues $\lm_k$ ($k={1,2,3,4}$) are
\begin{equation*}\label{}
\lm_1=\lm_2=0, \qquad \lm_3=-\lm_4=i \sqrt{\Dt},
\end{equation*}
where $i$ is the imaginary unit and $\Dt=b^2 + c^2 + d^2$.
Hence, 
the corresponding linearly independent eigenvectors $p_k$ ($k={1,2,3,4}$) are
\begin{equation*}\label{}
p_1(1,0,0,0), \quad p_2(0,b,c,d), \quad p_3(0,\al,\beta,\gamma),\quad
p_4(0,\overline{\alpha},\overline{\beta},\overline{\gamma}),
\end{equation*}
where $\al = c\Dt+ibd\sqrt{\Dt}$, $\beta = -b\Dt+icd\sqrt{\Dt}$, $\gamma = -i\sqrt{\Dt}(b^2+c^2)$
and $\overline{\al}$, $\overline{\beta}$, $\overline{\gamma}$ are the corresponding complex conjugates.
The vectors $p_k$ determine the following matrix 
\begin{equation*}\label{}
P=\left(
      \begin{array}{cccc}
        1 & 0 & 0 & 0\\
        0 & b & \alpha & \overline{\alpha}\\
        0 & c & \beta & \overline{\beta}\\
        0 & d & \gamma & \overline{\gamma}\\
      \end{array}
        \right)
\end{equation*}
having $\det P = -2i\Dt^{\frac52}(b^2+c^2)$.

Let us consider the first case when $\det P \neq 0$, \ie $(b,c) \neq (0,0)$.
Then we have
\begin{equation*}\label{}
P^{-1}=\left(
      \begin{array}{cccc}
        1 & 0 & 0 & 0 \\[6pt]
        0 & b\Dt^{-1} & c\Dt^{-1} & d\Dt^{-1} \\[6pt]
        0 & i\overline{\alpha}\delta & i\overline{\beta}\delta & i\overline{\gamma}\delta\\[6pt]
        0 & i\al \delta & i\beta \delta & i\gamma \delta\\
      \end{array}
        \right)
\end{equation*}
for the inverse matrix of $P$, where $\delta=\frac12\Dt^{-\frac32}\overline{\gamma}^{-1}$.
The Jordan matrix is the diagonal matrix $J=\mathrm{diag}(\lm_1,\lm_2,\lm_3,\lm_4)$. 
It is well known that
$e^A=Pe^JP^{-1}$.
Then we obtain the matrix representation of the respective Lie group of the considered Lie algebra in this case as follows
\begin{equation*}\label{}
e^A=\left(
      \begin{array}{cccc}
        1 & 0 & 0 & 0 \\
        0 & 1-(c^2+d^2)u & bcu+dt & bdu-ct \\
        0 & bcu-dt & 1-(b^2+d^2)u & cdu+bt\\
        0 & bdu+ct & cdu-bt & 1-(b^2+c^2)u\\
      \end{array}
    \right),
\end{equation*}
where $t=\Dt^{-\frac12}\sin\sqrt{\Dt}$ and $u=\Dt^{-1}\left(1-\cos\sqrt{\Dt}\right)$. This result can be written as
\begin{equation*}\label{}
e^A=E+tA+uA^2.
\end{equation*}

Now, let us consider the second case when $\det P = 0$, \ie $(b,c) = (0,0)$. Then $P$ is a non-invertible matrix.
Therefore $A$ is nilpotent and $e^A$ can be computed directly from
\begin{equation*}\label{}
e^A=E+A+\frac{A^2}{2!}+\frac{A^3}{3!}+\dots+\frac{A^{q-1}}{(q-1)!},
\end{equation*}
where $q$ stands for the nilpotency index of $A$.
Then, using the form of $A$ in \thmref{thm:main}, we obtain that $q=2$ in the considered case. Therefore, the matrix representation of the Lie group in this case is
$e^A=E+A$. In order to generalize both cases into the formula $e^A=E+tA+uA^2$, we substitute here $t=1$ and $u=0$.

By similar considerations we obtain the other results in \thmref{thm:main} for the rest of the classes.
\end{proof}

\section{Equipping of known Lie groups with hypercomplex structures and Hermitian-Norden metrics}\label{sect-ex}


In \cite{Kow}, there are considered the matrix Lie groups $G_6$, $G_8$ and $G_{10}$ of the following form
\begin{equation}\label{Kow6}
\begin{array}{c}
G_6=
\left(
      \begin{array}{cccc}
        1  & x & \frac{1}{2}x^2 & y \\
        0 & 1 & x & z\\
        0 & 0 & 1 & w\\
        0 & 0 & 0 & 1\\
      \end{array}
    \right),
    \end{array}
\end{equation}
\begin{equation}\label{Kow8}
\begin{array}{c}
   G_8=
\left(
      \begin{array}{cccc}
        \cos x  & \sin x & 0 & y \\
        -\sin x & \cos x & 0 & z\\
        z \cos x + y \sin x & z \sin x - y \cos x & 1 & w\\
        0 & 0 & 0 & 1\\
      \end{array}
    \right),
    \end{array}
\end{equation}
\begin{equation}\label{Kow10}
\begin{array}{c}
   G_{10}=
\left(
      \begin{array}{cccc}
        1  & x & y & z \\
        0 & 1 & w & v\\
        0 & 0 & 1 & w\\
        0 & 0 & 0 & 1\\
      \end{array}
    \right),
    \end{array}
\end{equation}
where $v,w,x,y,z\in\R$. These matrix groups and their automorphisms represent some of generalized affine symmetric spaces of infinite order which are considered in \cite{Kow}.

\begin{thm}
The matrix Lie groups $G_6$, $G_8$ and $G_{10}$ admit a hypercomplex structure with Hermitian-Norden metrics
of type $(hc4.1)$, $(hc3.2)$ and $(hc5.1)$, respectively
.
\end{thm}

\begin{proof}
Bearing in mind \thmref{thm:main} for $(hc4.1)$ ($t=1$, $u=a=b=c=0$), \eqref{Kow6} 
($x=z=w=0$) and substituting $y=d$, we have that $G_6$ can be considered as a hypercomplex manifold
with Hermitian-Norden metrics of type $(hc4.1)$.

Respectively, using \thmref{thm:main} for $(hc3.2)$
($t=1$, $u=a=b=0$, $c=-d=-1$) and \eqref{Kow8} 
($y=z=w=0$, $x=\frac{3\pi}{2}+2k\pi$, $k\in\mathbb{Z}$),
we have that $G_8$ can be considered as a hypercomplex manifold with Hermitian-Norden metrics of type $(hc3.2)$.

In a similar way, by means of \thmref{thm:main} for $(hc5.1)$ ($t=1$, $u=a=c=d=0$), \eqref{Kow10} 
($y=z=w=v=0$) and substituting $x=b$, we have that $G_{10}$ can be considered as a hypercomplex manifold with Hermitian-Norden metrics of type $(hc5.1)$.
\end{proof}


\begin{thebibliography}{33}

\bibitem{AlMa}
\textsc{D.\thinspace{}V.~Alekseevsky, S. Marchiafava},
\emph{Quaternionic structures
on a mani\-fold and subordinated structures}.
Ann. Mat. Pura Appl.
 \textbf{CLXXI} (IV) (1996) 205--273.





\bibitem{Barb}
\textsc{M.\thinspace{}L.~Barberis},
\emph{Hypercomplex structures on four-dimensional Lie groups}.
Proc. Amer. Math. Soc. \textbf{128} (4) (1997) 1043--1054.


\bibitem{Boy}
\textsc{C.\thinspace{}P.~Boyer},
\emph{A note on hyper-Hermitian four-manifolds}.
Proc. Amer. Math. Soc. \textbf{102} (1) (1988) 157--164.


\bibitem{GaBo}
\textsc{G. Ganchev, A. Borisov}, \emph{Note on the almost complex
manifolds with a Norden metric}. C. R. Acad. Bulg. Sci. \textbf{39} (1986) 31--34.

\bibitem{GaMiGr}
\textsc{G. Ganchev, V. Mihova, K. Gribachev}, \emph{Almost contact
manifolds with B-metric}. Math. Balkanica (N.S.) \textbf{7} (3-4) (1993) 261--276.

\bibitem{Gil}
\textsc{R. Gilmore}, \emph{Lie groups, Lie algebras and some of their applications}. A Wiley-Interscience Publication, New York (1974).


\bibitem{GrHe}
\textsc{A.~Gray, L.\thinspace{}M.~Hervella},
\emph{The sixteen classes of almost Hermitian manifolds and their linear invariants}.
Ann. Mat. Pura Appl. \textbf{CXXIII} (IV) (1980) 35--58.



\bibitem{GriMan24}
\textsc{K.~Gribachev, M.~Manev},
\emph{Almost hypercomplex pseudo-Her\-mit\-ian
manifolds and a 4-dimen\-sional Lie group with such structure}.
J. Geom. \textbf{88} (1-2) (2008) 41--52.


\bibitem{GriManDim12}
\textsc{K.~Gribachev, M.~Manev, S.~Dimiev},
\emph{On the almost hy\-per\-complex pseudo-Her\-mit\-ian manifolds}.
In:
Trends in Comp\-lex An\-al\-ysis, Differential Geometry and Mathe\-ma\-tical
Physics, eds. S. Dimiev and K. Sekigawa (World Sci. Publ.,
Singa\-pore, River Edge, NJ,  2003) 51--62.


\bibitem{Kow}
\textsc{O. Kowalski}, \emph{Generalized Symmetric Spaces, Lecture Notes in Mathematics}, 805 (Eds: A. Dold, B. Eckmann), Springer, Heidelberg (1980).



\bibitem{HM5}
\textsc{H.~Manev},
\emph{Matrix Lie groups as 3-dimensional almost contact B-metric manifolds}.
Facta Univ. Ser. Math. Inform., \textbf{30} (3) (2015) 341--351.





\bibitem{Ma05}
\textsc{M.~Manev},
\emph{Tangent bundles with Sasaki metric and almost
hypercomplex pseudo-Hermitian structure}.
In: Topics in Almost Hermitian Geometry and Related Fields.
(Y. Matsushita, E. Garcia-Rio, H. Hashimoto, T. Koda and T. Oguro, eds.) World Sci. Publ., Singapore (2005) 170--185.

\bibitem{Man28}
\textsc{M.~Manev},
\emph{A connection with parallel torsion on almost
hypercomplex manifolds with Hermitian and anti-Hermitian metrics}.
J. Geom. Phys. \textbf{61} (1) (2011) 248--259.



\bibitem{Man44} 
\textsc{M. Manev}, \emph{Hypercomplex structures with Hermitian-Norden metrics on four-dimen\-sional Lie algebras},
J. Geom. \textbf{105} (1) (2014) 21--31.

\bibitem{ManGri32}
\textsc{M.~Manev, K.~Gribachev},
\emph{A connection with parallel totally skew-symmetric torsion on a class
of almost hypercomplex manifolds with Hermitian and anti-Hermitian metrics}.
Int. J. Geom. Methods Mod. Phys. \textbf{8} (1) (2011) 115--131.


\bibitem{ManSek}
\textsc{M.~Manev, K.~Sekigawa},
\emph{Some four-dimensional almost hypercomplex pseudo-Hermitian manifolds}.
In: Contemporary Aspects of Complex Analysis, Differential Geometry and Mathematical
Physics (S.~Di\-miev and K.~Sekigawa, eds.) World
Sci. Publ., Singapore (2005) 174--186.




\bibitem{NakHMan}
\textsc{G.~Nakova, H.~Manev},
\emph{Holomorphic submanifolds of some hypercomplex manifolds
with Hermitian and Norden metrics}.
C. R. Acad. Bulgare Sci. \textbf{70} (1) (2017) 29--40.


\bibitem{So}
\textsc{A.~Sommese},
\emph{Quaternionic manifolds}.
Math. Ann. \textbf{212} (1975) 191--214.

\bibitem{ZamNak}
\textsc{S. Zamkovoy, G. Nakova}. \emph{The decomposition of almost paracontact metric manifolds in eleven classes revisited}.
J. Geom. 109:18 (2018)

\end{thebibliography}
\end{document}